\documentclass[10pt]{amsart}
\usepackage{amssymb,graphicx,epsfig,url,setspace,color,comment,pdflscape}
\theoremstyle{plain}

\makeatletter
\@namedef{subjclassname@2020}{%
  \textup{2020} Mathematics Subject Classification}
\makeatother
\newtheorem{thm}{Theorem}[section]
\newtheorem{cor}[thm]{Corollary}
\newtheorem{df}[thm]{Definition}
\newtheorem{lem}[thm]{Lemma}

\newtheorem{rem}[thm]{Remark}
\newtheorem{ques}[thm]{Question}
\newtheorem{conj}[thm]{Conjecture}

\def\cal{\mathcal}
\def\bbb{\mathbb}

\renewcommand{\phi}{\varphi}

\newcommand{\N}{\bbb{N}}
\newcommand{\Z}{\bbb{Z}}

\newcommand{\bs}{\backslash}

\begin{document}
\title[]{On Frobenius problem with restrictions on common divisors of coefficients}

\author{Piotr Miska}
\address{Institute of
Mathematics, 
Faculty of Mathematics and Computer Science, 
Jagiellonian University in Cracow, 
{\L}ojasiewicza 6, 
30-348 Krak\'ow, Poland; 
email: \texttt{piotr.miska@uj.edu.pl}}

\author{Maciej Zakarczemny}
\address{Cathedral of Applied Mathematics, 
Faculty of Computer Science and Telecomunication, 
Cracow University of Technology, 
Warszawska 24, 
31-155 Krak\'ow, Poland; 
email: \texttt{mzakarczemny@pk.edu.pl}}

\thanks{The research of the first author is supported by the grant of the Polish National Science Centre (NCN), Poland, no. UMO-2019/34/E/ST1/00094}

\begin{abstract}
Let $m,s,t$ are positive integers with $t\leq s-2$ and $a_1,a_2,\ldots,a_s$ are positive integers such that $(a_1,a_2,\ldots,a_{s-1})=1$. In the paper we prove that every sufficiently large positive integer can be written in the form $a_1\mu_1+a_2\mu_2+\ldots+a_s\mu_m$, where positive integers $\mu_1,\mu_2,\ldots,\mu_s$ have no common divisor being $m$-th power of a positive integer greater than $1$ but each $t$ of the values of $\mu_1,\mu_2,\ldots,\mu_n$ have a common divisor being $m$-th power of a positive integer greater than $1$. Moreover, we show that every sufficiently large positive integer can be written as a sum of positive integers $\mu_1,\mu_2,\ldots,\mu_n$ with no common divisor being $m$-th power of a positive integer greater than $1$ but each $s-1$ of the values of $\mu_1,\mu_2,\ldots,\mu_s$ have a common divisor being $m$-th power of a positive integer greater than $1$.
\end{abstract}

\maketitle

\section{Introduction}\label{section1}
Let $\Z_{\ge 0}$ denote the set of all non-negative integers and $\N$ or $\Z_{\ge 1}$ denote the set of all positive integers. Let $\mathbb{P}$ denote the set of prime numbers. If the letter $A$ stands for a subsequence of $\mathbb{Z}_{\ge 0}$ and there exists a natural number $h$ such that $\underset{h\text{ times}}{\underbrace{A+A+\ldots+A}}=hA=\mathbb{Z}_{\ge 0},$ then $A$ is called a basis of order $h$, see \cite{HR}. Thus the sequence of all odd numbers with zero is a basis of order $2$. The sequence of nonnegative squares is a basis of order $4$ by Lagrange's four-square theorem, but not a basis of order $3$ by Gauss-Legendre's three-square theorem. It is known (see \cite{HR}) that every sequence which contains zero and has positive Schnirelmann density is a basis. Schnirelmann proved that the sequence consisting of $0,1$ and the primes is a basis. At the time, this result was a contribution to unsolved Goldbach problem. In $1943$ Linnik (see \cite{LIN}) gave an elementary proof of the Waring-Hilbert theorem that, for each natural $k$, the sequence of nonnegative $k$-powers is a basis.

The famous Frobenius coin problem (see \cite{RA,RKG}) can be stated as follows: for given positive relatively prime positive integers $a_1,a_2,\ldots, a_n$ find the largest integer $g(a_1,a_2,\ldots,a_n)$, called the Frobenius number, that is not in the sum of arithmetic sequences $a_1\mathbb{Z}_{\ge 0}+a_2\mathbb{Z}_{\ge 0}+\ldots +a_n\mathbb{Z}_{\ge 0}$. In general, the question is asked about representation of positive integer as a sum of sequences $A_1,A_2,\ldots,A_n\subseteq \mathbb{Z}_{\ge 0}$ of certain types. In our paper, instead of focusing on the form of the summands in the representations of positive integers as sums, we are interested in the relations among these summands. Namely, we consider representation of positive integer as a sum $a_1\mu_1+a_2\mu_2+\ldots+a_s\mu_s$, where $\mu_1,\mu_2,\ldots,\mu_s\in\mathbb{N}$ with extra conditions on common divisors of coefficients $\mu_1,\mu_2,\ldots,\mu_s\in\Z_{\ge 0}$.

For notational simplicity, we denote $[n]:=\{1,2,\ldots,n\}.$
Writing $p_n$ we mean the $n$-th prime number. We use standard notation $(a_1,a_2,\ldots ,a_k)=\mathrm{GCD}(a_1,a_2,\ldots ,a_k).$ If $\phi$ is a logic sentence, then we put
$$\delta(\phi)=\begin{cases}1\mbox{ if }\phi\mbox{ is true},\\0\mbox{ if }\phi\mbox{ is false}.\end{cases}$$

One of the results of \cite{MiMuSa}, namely Proposition 1.5, comes down to checking that $31$ is the least positive integer which can be written as a sum of positive integers $\mu_1,\ldots ,\mu_s$, $s\geq 3$, 
such that $(\mu_1,\ldots ,\mu_s)=1$ but $(\mu_i,\mu_j)>1$ for any $i,j\in [s]$. 
Then, in a~natural way, there emerges a~problem of giving all the positive integers $n$ such that $n=\sum_{i=1}^s\mu_i$ for some $s\geq 3$ and $\mu_1,\ldots ,\mu_s\in\N$ 
such that $(\mu_1,\ldots  ,\mu_s)=1$ and $(\mu_i,\mu_j)>1$ for any $i,j\in [s]$. 
Let us denote the set of all the positive integers $n$ with this property as $\cal{S}$. Fixing an integer $s\geq 3$ we put
$$\cal{S}_s=\left\{\sum_{i=1}^s\mu_i: \mu_1,\ldots  ,\mu_s\in\N, (\mu_1,\ldots  ,\mu_s)=1\mbox{ and }(\mu_i,\mu_j)>1\mbox{ for each }i,j\in [s]\right\}.$$
Then, of course, $\cal{S}=\bigcup_{s=3}^{\infty}\cal{S}_s$.

We can also consider the following sets:
$$\cal{S}_{s,t}=\left\{\sum_{i=1}^s\mu_i: \mu_1,\ldots ,\mu_s\in\N, (\mu_{i_1},\ldots ,\mu_{i_t})>1\mbox{ for any }1\leq i_1<\ldots <i_t\leq s\mbox{ and }(\mu_1,\ldots ,\mu_s)=1\right\},$$
where $s,t\in\N$, $s\geq 3$ and $1\leq t\leq s-1$. We see that $\cal{S}_{s,2}=\cal{S}_s$. \\
Let $a_1,a_2,\ldots,a_s$ be positive integers such that $(a_1,a_2,\ldots,a_s)=1$ and denote by 
$$\cal{S}_s(a_1,a_2,\ldots,a_s)=\left\{\sum_{i=1}^sa_i\mu_i: \mu_1,\ldots  ,\mu_s\in\N, (\mu_1,\ldots  ,\mu_s)=1\mbox{ and }(\mu_i,\mu_j)>1\mbox{ for each }i,j\in [s]\right\}.$$

We give the definition of the generalized greatest $m$-th power common divisor, see \cite{ECoh,KVN}.
\begin{df}
Let $a_1,a_2,\ldots,a_k\in \mathbb{Z},$ not all equal to zero and $m\in\mathbb{N}$. Then $(a_1,a_2,\ldots,a_k)_m$ is the largest divisor of the numbers $a_1,\ldots,a_k$ of the form $d^m$, where $d\in\mathbb{N}$, i.e. $d^m\mid a_i$ for all $i\in [k]$.
 
\end{df}
\begin{rem}
{\rm In the above notation $(a,b)=(a,b)_1.$ If $(a,b)_2=1$ then $(a,b)$ is a squarefree number. 
Note that, $(a,b)_m=1$ if and only if all prime numbers in the decomposition of $(a,b)$ appear with exponents smaller than $m$.}
\end{rem}
By analogy with the notation of $\cal{S}_s$ and $\cal{S}_s(a_1,a_2,\ldots,a_s)$ we define: 
$$\cal{S}^m_s=\left\{\sum_{i=1}^s\mu_i: \mu_1,\ldots  ,\mu_s\in\N, (\mu_1,\ldots  ,\mu_s)_m=1\mbox{ and }(\mu_i,\mu_j)_m>1\mbox{ for each }i,j\in [s]\right\},$$
$$\cal{S}^m_s(a_1,a_2,\ldots,a_s)=\left\{\sum_{i=1}^s a_i\mu_i: \mu_1,\ldots  ,\mu_s\in\N, (\mu_1,\ldots  ,\mu_s)_m=1\mbox{ and }(\mu_i,\mu_j)_m>1\mbox{ for each }i,j\in [s]\right\}$$
and
\begin{align*}
\cal{S}^m_{s,t}(a_1 & ,a_2,\ldots,a_s)=\\
& \left\{\sum_{i=1}^s a_i\mu_i: \mu_1,\ldots  ,\mu_s\in\N, (\mu_1,\ldots  ,\mu_s)_m=1\mbox{ and }(\mu_{i_1},\ldots ,\mu_{i_t})_m>1\mbox{ for any }1\leq i_1<\ldots <i_t\leq s\right\},
\end{align*}
where $s,t\in\N$, $s\geq 3$ and $1\leq t\leq s-1$.
\begin{rem}
Note that
\begin{align*}
\cal{S}^2_{s,1}(a_1 & ,a_2,\ldots,a_s)=\\
& \left\{\sum_{i=1}^s a_i\mu_i: \mu_1,\ldots  ,\mu_s\in\N, \mbox{ where }\mu_{i_1},\ldots ,\mu_{i_t}\mbox{ are not squarefree numbers and }(\mu_1,\ldots  ,\mu_s)\mbox{ is squarefree}\right\}.
\end{align*}
\end{rem}

Our main result is that the set $\cal{S}_{s,t}^m(a_1,a_2,\ldots,a_s)$, $s\ge 3$, $t<s$, contains all the positive integers but finitely many exceptions on condition that $(a_1,a_2,\ldots,a_{s-1})=1$ and $t\le s-2$ or $a_1=\ldots =a_s=1$. In particular, for each positive integers $s>t$ with $s\geq 3$ all but finitely many positive integers can be written as sums of $s$ positive integers with greatest common divisor not divisible by any perfect $m$-th power greater than $1$ but any $t$ of summands have common divisor being a perfect $m$-th power greater than $1$.

Let us describe the content of the paper. In Section 2 we formulate the statements of our results. Section 3 contains lemmas that are subsidiary in the proofs of our main results, given in Sections 4, 5 and 6. In Section 7 we state some open problems and possible directions of further study in the area of Frobenius problem with restrictions.

\section{Main results}

\begin{thm}\label{Ss}
Let $s\ge 3.$ Let $a_k$, where $k\in [s]$, be positive integers such that $(a_1,a_2,\ldots,a_{s-1})=1.$ Let $p_{\{i,j\}}$, where $i,j\in [s],\,\,i\neq j,$ be pairwise 
distinct $\genfrac(){0pt}{2}{s}{2}$ prime numbers such that $p_{\{i,j\}}\nmid (a_i,a_j)$ for each $i,j\in [s-1],\,\,i\neq j,$ and $p_{\{i,s\}}\nmid a_i$ for each $i\in [s-1]$. Then any positive integer
\begin{align*}
n\ge & a_{s-1}\prod\limits_{\{i,j\}\subset [s]\bs\{s-1\},\,i\neq j}p_{\{i,j\}}^m+a_s\prod\limits_{\{i,j\}\subset [s-1],\,i\neq j}p_{\{i,j\}}^m\\
& +\left(\prod_{\{i,j\}\subset [s],\,i\neq j}p_{\{i,j\}}^m\right)\cdot\sum_{k=1}^{s-2}\left(\frac{a_ka_{s-1}}{(a_k,a_{s-1})}\cdot\frac{2}{p_{\{k,s-1\}}^m}+\sum_{l\in [s-2],\,l>k}\frac{a_ka_l}{(a_k,a_l)}\cdot\frac{1}{p_{\{k,l\}}^m}\right)
\end{align*}
belongs to the set $S_{s,s-2}^m(a_1,a_2,\ldots,a_s)$ and thus, to the set $S_{s,t}^m(a_1,a_2,\ldots,a_s)$ for each $t\in [s-2]$.
\end{thm}

\begin{cor}\label{NNN1}
Let $a_1,a_2,a_3,a_4$ be positive integers such that $(a_1,a_2,a_3)=1.$
Let $p_{1,2}$, $p_{1,3}$, $p_{1,4}$, $p_{2,3}$, $p_{2,4}$, $p_{3,4}$ be pairwise distinct prime numbers such that:
\begin{align*}
&p_{1,4}\nmid a_1,\,\,p_{2,4}\nmid a_2,\,\,p_{3,4}\nmid a_3,\\
&p_{1,2}\nmid (a_1,a_2),\,\,p_{1,3}\nmid (a_1,a_3),\,\,p_{2,3}\nmid (a_2,a_3).
\end{align*}
Then any positive integer 
\begin{align*}
n&\ge a_3 p_{1,2}^m p_{1,4}^m p_{2,4}^m + a_4 p_{1,2}^m p_{1,3}^m p_{2,3}^m + \tfrac{a_1a_2}{(a_1,a_2)} p_{1,3}^m p_{1,4}^m p_{2,3}^m p_{2,4}^m p_{3,4}^m\\
& + 2 \tfrac{a_1 a_3}{(a_1,a_3)} p_{1,2}^m p_{1,4}^m p_{2,3}^m p_{2,4}^m p_{3,4}^m + 2 \tfrac{a_2 a_3}{(a_2,a_3)} p_{1,2}^m p_{1,3}^m p_{1,4}^m p_{2,4}^m p_{3,4}^m
\end{align*}
belongs to the set $\cal{S}^m_4(a_1,a_2,a_3,a_4).$ In other words, $n=a_1\mu_1+a_2\mu_2+a_3\mu_3+a_4\mu_4,$ where $\mu_1,\mu_2,\mu_3,\mu_4\in\N$ with additional restrictions $(\mu_1,\mu_2,\mu_3,\mu_4)_m=1$ and $(\mu_i,\mu_j)_m>1$ for each $i,j\in\{1,2,3,4\}.$
\end{cor}

We get immediately from Corollary \ref{NNN1}.

\begin{cor}\label{S4prim}
Let $p_{1,2}$, $p_{1,3}$, $p_{1,4}$, $p_{2,3}$, $p_{2,4}$, $p_{3,4}$ be pairwise distinct prime numbers. Then any positive integer 
\begin{align*}
n&\ge p_{1,2}^m p_{1,4}^m p_{2,4}^m + p_{1,2}^m p_{1,3}^m p_{2,3}^m + p_{1,3}^m p_{1,4}^m p_{2,3}^m p_{2,4}^m p_{3,4}^m + 2 p_{1,2}^m p_{1,4}^m p_{2,3}^m p_{2,4}^m p_{3,4}^m + 2 p_{1,2}^m p_{1,3}^m p_{1,4}^m p_{2,4}^m p_{3,4}^m
\end{align*}
belongs to the set $\cal{S}^m_4.$ i.e. $n=\mu_1+\mu_2+\mu_3+\mu_4,$ where $\mu_1,\mu_2,\mu_3,\mu_4,$ have no common divisor which is larger than 
$1$ and is a perfect $m$-th power of an integer but for each $i,j\in\{1,2,3,4\}$ the numbers $\mu_i,\mu_j$ have a common divisor being a perfect $m$-th power larger than $1$.
\end{cor}

\begin{cor}
Let $a_1,a_2,a_3,a_4$ be positive integers such that $(a_1,a_2,a_3)=1.$
Let $p_{1,2}$, $p_{1,3}$, $p_{1,4}$, $p_{2,3}$, $p_{2,4}$, $p_{3,4}$ be pairwise distinct prime numbers such that:
\begin{align*}
&p_{1,4}\nmid a_1,\,\,p_{2,4}\nmid a_2,\,\,p_{3,4}\nmid a_3,\\
&p_{1,2}\nmid (a_1,a_2),\,\,p_{1,3}\nmid (a_1,a_3),\,\,p_{2,3}\nmid (a_2,a_3).
\end{align*}
Then any positive integer 
\begin{align*}
n&\ge a_3 p_{1,2}^2 p_{1,4}^2 p_{2,4}^2 + a_4 p_{1,2}^2 p_{1,3}^2 p_{2,3}^2 + \tfrac{a_1a_2}{(a_1,a_2)} p_{1,3}^2 p_{1,4}^2 p_{2,3}^2 p_{2,4}^2 p_{3,4}^2\\
& + 2 \tfrac{a_1 a_3}{(a_1,a_3)} p_{1,2}^2 p_{1,4}^2 p_{2,3}^2 p_{2,4}^2 p_{3,4}^2 + 2 \tfrac{a_2 a_3}{(a_2,a_3)} p_{1,2}^2 p_{1,3}^2 p_{1,4}^2 p_{2,4}^2 p_{3,4}^2
\end{align*}
is of the form $n=a_1\mu_1+a_2\mu_2+a_3\mu_3+a_4\mu_4,$ 
where $\mu_1,\mu_2,\mu_3,\mu_4$ are natural numbers such that $(\mu_1,\mu_2,\mu_3,\mu_4)$ is a squarefree natural number, 
but for each $i,j\in\{1,2,3,4\},$ the numbers $(\mu_i,\mu_j)$ are not a squarefree natural numbers.  
\end{cor}

\begin{rem}
{\rm By taking in the corollary above $a_1=a_2=a_3=a_4=1$, $p_{1,2}=5$, $p_{1,3}=3$, $p_{1,4}=13$, $p_{2,3}=2$, $p_{2,4}=11$, $p_{3,4}=7$, we obtain the following statement: 
If $n\ge 44985865$, then $n=\mu_1+\mu_2+\mu_3+\mu_4,$ where $\mu_1,\mu_2,\mu_3,\mu_4$ 
are natural numbers such that $(\mu_1,\mu_2,\mu_3,\mu_4)$ is a squarefree natural number 
but for each $i,j\in\{1,2,3,4\},\, i\neq j,$ the numbers $(\mu_i,\mu_j)$ are not a squarefree natural numbers.}
\end{rem}

\begin{cor}
Let $a_1,a_2,a_3,a_4$ be positive integers such that $(a_1,a_2,a_3)=1.$
Let $p_{1,2}$, $p_{1,3}$, $p_{1,4}$, $p_{2,3}$, $p_{2,4}$, $p_{3,4}$ be pairwise distinct prime numbers such that:
\begin{align*}
&p_{1,4}\nmid a_1,\,\,p_{2,4}\nmid a_2,\,\,p_{3,4}\nmid a_3,\\
&p_{1,2}\nmid (a_1,a_2),\,\,p_{1,3}\nmid (a_1,a_3),\,\,p_{2,3}\nmid (a_2,a_3).
\end{align*}
Then any positive integer 
\begin{align*}
n&\ge a_3 p_{1,2} p_{1,4} p_{2,4} + a_4 p_{1,2} p_{1,3} p_{2,3} + \tfrac{a_1a_2}{(a_1,a_2)} p_{1,3} p_{1,4} p_{2,3} p_{2,4} p_{3,4}\\
& + 2 \tfrac{a_1 a_3}{(a_1,a_3)} p_{1,2} p_{1,4} p_{2,3} p_{2,4} p_{3,4} + 2 \tfrac{a_2 a_3}{(a_2,a_3)} p_{1,2} p_{1,3} p_{1,4} p_{2,4} p_{3,4}
\end{align*}
is of the form $n=a_1\mu_1+a_2\mu_2+a_3\mu_3+a_4\mu_4,$ where $\mu_1,\mu_2,\mu_3,\mu_4\in\N$, 
with additional restrictions\linebreak $(\mu_1,\mu_2,\mu_3,\mu_4)~=~1$ and $(\mu_i,\mu_j)>1$ for each $i,j\in\{1,2,3,4\}.$
\end{cor}

\begin{rem}
If we take $a_1=a_2=a_3=a_4=1$, $p_{1,2}=7$, $p_{1,3}=11$, $p_{1,4}=2$, $p_{2,3}=13$, $p_{2,4}=3$, $p_{3,4}=5$ in the above corollary, then we obtain the following statement: If $n_0=15413$, then $n=\mu_1+\mu_2+\mu_3+\mu_4,$ where $\mu_1,\mu_2,\mu_3,\mu_4$ 
are natural numbers such that $(\mu_1,\mu_2,\mu_3,\mu_4)=1$
but for each $i,j\in\{1,2,3,4\},\, i\neq j,$ we have $(\mu_i,\mu_j)>1$.
\end{rem}

\begin{rem}
{\rm In case of $s=3$, $m=1$, $a_1=a_2=a_3=1$ and $p_{\{1,2\}}=5$, $p_{\{1,3\}}=2$, $p_{\{2,3\}}=3$, Theorem \ref{Ss} states that each positive integer $n\ge 19$ can be written as a sum of three integers $>1$ such that their greatest common divisor is $1$. However, it is easy to show that $\cal{S}_{3,1}=\N\bs\{1,2,3,4,5,6\}$ as every $n\geq 7$ can be written as $2+3+(n-5)$, where $n-5>1$.}
\end{rem}

\begin{rem}
{\rm In case of $s=3$, $m=2$, $a_1=a_2=a_3=1$ and $p_{\{1,2\}}=5$, $p_{\{1,3\}}=2$, $p_{\{2,3\}}=3$, Theorem \ref{Ss} states that each positive integer $n\ge 101$ can be written as a sum of three non-squarefree integers such that their greatest common divisor is squarefree.}
\end{rem}

\begin{thm}\label{S3}
If $n\ge 26 177 082$, then $n\in\cal{S}_3$.
\end{thm}

\begin{rem}
Spurred by some computer calculations, we have
\begin{align*}
\N\bs\cal{S}_3\supseteq\{ & 2730, 2310, 1680, 1470, 1320, 1260, 1140, 1050, 990, 930, 924, 840, 780, 750, 720, 690, 672, 660, 630, 600,\\
& 546, 540, 510, 504, 480, 474, 462, 450, 420, 390, 378, 372, 360, 348, 336, 330, 306, 300, 294, 288, 280, 276,\\
& 270, 264, 258, 252, 246, 240, 234, 228, 220, 216, 210, 204, 200, 198, 192, 186, 180, 174, 170, 168, 165, 162,\\
& 160, 156, 154, 150, 144, 142, 140, 138, 132, 130, 126, 124, 120, 114, 112, 110, 108, 105, 104, 102, 100, 98,\\
& 96, 94, 90, 88, 84, 82, 81, 80, 78, 76, 75, 74, 72, 70, 68, 66, 64, 63, 62, 60, 58, 57, 56, 54, 52, 51, 50, 48, 46,\\
& 45, 44, 42, 40, 39, 38, 36, 35, 34, 33, 32, 30, 29, 28, 27, 26, 25, 24, 23, 22, 21, 20, 19, 18, 17, 16, 15, 14, 13,\\
& 12, 11, 10, 9, 8, 7, 6, 5, 4, 3, 2, 1\}
\end{align*}
\end{rem}

\begin{thm}\label{Sss-1}
For each integers $s,m$ with $s\geq 3$ the set $\cal{S}_{s,s-1}^m$ contains all the positive integers at least equal to $(s+1)\left(\prod_{j=1}^{c_{s,m}-1}p_j\right)\left(\prod_{i=1}^{s}p_{c_{s,m}-i}^{m-1}\right)$, where $c_{s,m}\in\N$ is such that $2p_{j-ms}>p_j$ for each $j\geq c_{s,m}-s+1$ and $p_{c_{s,m}-(m+1)s}>2^{(2m+1)s+1}$.
\end{thm}

\begin{rem}
The inequality $2p_{j-s}>p_j$ holds for $j\gg 0$ as $$\pi(2x)-\pi(x)\sim\frac{2x}{\log 2x}-\frac{x}{\log x}\sim\frac{2x}{\log x}-\frac{x}{\log x}=\frac{x}{\log x}\to\infty,$$ when $x\to\infty$.
\end{rem}


\section{Auxiliary results}

The proofs of the above theorems are based on the following results. The first one is an improvement of \mbox{\cite[Theorem 1.16,\,p. 38]{MBN}}.

\begin{lem}\label{F1}
  Let $a_1,\ldots ,a_k$ be positive integers such that $(a_1,\ldots ,a_k)=1.$ 
\begin{enumerate}
\item[a)]If $b\ge \sum\limits_{i=1}^{k-1}a_i\left(\frac{a_k}{(a_k,a_i)}-1\right),$ then there exist integers $x_1,\ldots ,x_k\ge 0$ such that $a_1x_1+\ldots +a_kx_k=b.$
\item[b)] If $b\ge \sum\limits_{i=1}^{k}\frac{a_k a_i}{(a_k,a_i)},$ then there exist integers $x_1,\ldots ,x_k>0$ such that $a_1x_1+\ldots +a_kx_k=b.$
\end{enumerate}
\end{lem}
\begin{proof}
There exist integers $z_1,\ldots ,z_k$ such that $a_1z_1+\ldots +a_kz_k=b,$
since $(a_1,\ldots ,a_k)=1.$ We divide each of the integers $z_1,\ldots ,z_{k-1}$ by $\frac{a_k}{(a_k,a_i)}$ and get:
\begin{equation*}
z_i=\frac{a_k}{(a_k,a_i)}q_i+x_i,\,\,0\le x_i\le \frac{a_k}{(a_k,a_i)}-1,\,\, i\in [k-1].
\end{equation*}
Let $x_k=z_k+\sum\limits_{i=1}^{k-1}\frac{a_i}{(a_k,a_i)}q_i.$ Then
\begin{align*}
b&=a_1z_1+\ldots +a_{k-1}z_{k-1}+a_kz_k=\\
&=a_1\left(\frac{a_k}{(a_k,a_1)}q_1+x_1\right)+\ldots +a_{k-1}\left(\frac{a_k}{(a_k,a_{k-1})}q_{k-1}+x_{k-1}\right)+a_kz_k=\\
&=a_1x_1+\ldots +a_{k-1}x_{k-1}+a_k\left(z_k+\sum\limits_{i=1}^{k-1}\frac{a_i}{(a_k,a_i)}q_i\right)\\
&=a_1x_1+a_2x_2+\ldots +a_{k-1}x_{k-1}+a_kx_k\le \sum\limits_{i=1}^{k-1}a_i\left(\frac{a_k}{(a_k,a_i)}-1\right) +a_kx_k.
\end{align*}
If $b\ge \sum\limits_{i=1}^{k-1}a_i\left(\frac{a_k}{(a_k,a_i)}-1\right) $
then $a_kx_k\ge 0$ and so $x_k\ge 0.$ This completes the proof of $a).$\\
If $b\ge \sum\limits_{i=1}^{k}\frac{a_k a_i}{(a_k,a_i)},$ then $b-\sum\limits_{i=1}^{k}a_i\ge \sum\limits_{i=1}^{k-1}a_i\left(\frac{a_k}{(a_k,a_i)}-1\right).$ By part $a)$ we know that there are $x_1,\ldots ,x_k\ge 0$ such that $a_1x_1+\ldots +a_kx_k=b-\sum\limits_{i=1}^{k}a_i.$ Therefore $a_1(x_1+1)+\ldots +a_k(x_k+1)=b$ and $x_i+1>0.$ This completes the proof of $b).$
\end{proof}

The next two lemmas will be crucial in the proof of Theorem \ref{S3}. From now on, by $p_j$ we mean the $j$-th prime number with respect to the indexing prime numbers in ascending order. Moreover, we put $p_j=1$ for $j\in\Z\bs\N$ and set a convention that the product over empty set is equal to $1$.

\begin{lem}\label{F2}
Let $p_j, p_k, p_l$ be prime numbers, where $j<k<l$. Then, a positive integer $n$ coprime to $p_jp_kp_l$ can be written in the form $p_jp_kx+p_jp_ly+p_kp_lz$ with positive integers $x,y,z$ such that $(x,y,z)=1, (p_j,z)=(p_k,y)=(p_l,x)=1$ on condition that:
\begin{enumerate}
\item $j\geq 10$ and $n\ge (8p_1\cdot\ldots\cdot p_{j-6}+6)p_j p_k p_l$,
\item $j\in\{8,9\}$ and $n\geq (6p_1\cdot\ldots\cdot p_{j-5}+6)p_j p_k p_l$,
\item $s\in\{6,7\}$ and $n\geq (4p_1\cdot\ldots\cdot p_{j-4}+6)p_j p_k p_l$,
\item $s=5$ and $n\geq (3p_1p_2+6)p_j p_k p_l$,
\item $s=4$ and $n\geq (2p_1p_2+6)p_j p_k p_l$,
\item $s=3$ and $n\geq 6p_j p_k p_l$,
\item $s=2$ and $n\geq 4p_j p_k p_l$,
\item $s=1$ and $n\geq 3p_j p_k p_l$.
\end{enumerate}
\end{lem}
\begin{proof}
Let $n=p_jp_kx+p_jp_ly+p_kp_lz$ for some $x,y,z\in\N$.\\
At first, let us note that for any value of $j\in\N$ we can assume $z<p_j.$ If $z\ge p_j,$ then we replace
\begin{equation}\label{Subs1}
z\mapsto z-\left\lfloor\frac{z}{p_j}\right\rfloor p_j>0,\, x\mapsto x+\left\lfloor\frac{z}{p_j}\right\rfloor p_l,
\end{equation}
since $(p_j,z)=1.$
\bigskip

Now, we focus on the case of $j\geq 10$.

We can modify the coefficients $x,y$ such that $(x,y,6)=1.$ Indeed, we replace
\begin{equation}\label{Subs2}
\begin{cases}
x\mapsto x+p_l,\,y\mapsto y-p_k, & \mbox{if} \quad 6|(x,y,z),\\
x\mapsto x+2p_l,\,y\mapsto y-2p_k, & \mbox{if} \quad 3\mid (x,y,z),\, 2\nmid(x,y,z),\\
x\mapsto x+3p_l,\,y\mapsto y-3p_k, & \mbox{if} \quad 2\mid (x,y,z),\, 3\nmid(x,y,z).
\end{cases}
\end{equation}
Let $J_j=\{p:p\le p_{j-6},p\nmid(x,y,z),\,p\in\mathbb{P}\}$.\\
Note that $J_j\supseteq\{2,3\}$ and $\{p:p\le p_{j-6}:\,p\in\mathbb{P}\}\supseteq\{2,3,5,7\}.$\\
If $(x,y,z)>1,$ then we can replace
\begin{equation}\label{Subs3}
x\mapsto x+Cp_l\prod\limits_{p\in J_j}p,\,y\mapsto y-Cp_k\prod\limits_{p\in J_j}p,
\end{equation}
where $C\in\{1,2,3,4,6,8\}$ such that $(p_{j-5}\cdot\ldots\cdot p_{j-1},x+Cp_l\prod\limits_{p\in J_j}p)=1.$\\
The last substitution can be done because:
\begin{itemize}
\item if $C,C'\in \{1,2,3,4,6,8\},\, C\neq C'$, then $p_i>7\ge |C-C'|$ for $j-5\le i\le j-1$; thus $C\not\equiv C' \pmod {p_i}$;
\item for each $j-5\le i\le j-1$ there exists at most one $C\in\{1,2,3,4,6,8\}$ such that $p_i\mid x+Cp_l\prod\limits_{p\in J_j}p$;
\item therefore, there exists at least one $C\in\{1,2,3,4,6,8\}$ such that $p_i\nmid x+Cp_l\prod\limits_{p\in J_j}p$ for any $j-5\le i\le j-1$.
\end{itemize}
Now we will explain that after all these modifications we have $\left(x,y,z\right)=1$.
\begin{itemize}
\item Since $z<p_j$, all the prime divisors of $z$ are less than $p_j$ and thus all the prime divisors of $(x,y,z)$ are less than $p_j$.
\item The substitution (\ref{Subs2}) ensures us that $2$ and $3$ are not divisors of $(x,y,z)$.
\item As $2,3\in J_j$, after substitution (\ref{Subs3}) the numbers $2$ and $3$ still are not divisors of $(x,y,z)$.
\item If $p\in J_j$, then substitution (\ref{Subs3}) does not change divisibility of $x$ and $y$ by $p$. Hence $p\nmid (x,y,z)$.
\item If $p\not\in J_j$ and $p<p_{j-5}$, then $p\mid (x,y,z)$ before substitution (\ref{Subs3}). Because $p\nmid Cp_kp_l\prod\limits_{q\in J_j}q$, the substitution (\ref{Subs3}) changes the divisiblity of $x$ and $y$ by $p$, i.e. $p\nmid x$ and $p\nmid y$ after substitution (\ref{Subs3}).
\item If $p=p_i,$ where $j-5\leq i\leq j-1$, then $p_i\nmid x$ after substitution (\ref{Subs3}). This follows directly from the choice of $C$.
\end{itemize}
If the beginning value of $y$ is greater than $p_k(8p_1\cdot\ldots\cdot p_{j-6}+3)$, then it remains positive after substitutions (\ref{Subs2}) and (\ref{Subs3}). Hence, if $n-p_jp_kp_l(8p_1\cdot\ldots\cdot p_{j-6}+3)\ge 3p_j p_k p_l$, then by Lemma \ref{F1} b) the number $n-p_jp_kp_l(8p_1\cdot\ldots\cdot p_{j-6}+3)$ can be written in the form $p_jp_k\tilde{x}+p_jp_l\tilde{y}+p_kp_l\tilde{z}$ for some $\tilde{x},\tilde{y},\tilde{z}\in\N$. Thus, $n$ can be written in the form $p_jp_kx+p_jp_ly+p_kp_lz$ for some $x,z>0$ and $y>p_k(8p_1\cdot\ldots\cdot p_{j-6}+3)$. Hence, applying substitutions (\ref{Subs1})--(\ref{Subs3}), we conclude that $n$ can be written as $p_jp_kx+p_jp_ly+p_kp_lz$ for some $x,y,z\geq 0$ with $(x,y,z)=1$. The equality $(p_j,z)=(p_k,y)=(p_l,x)=1$ follows directly from the assumption $(n,p_j p_k p_l)=1$.

\bigskip
Since explanations in the remaining cases are analogous, we will only give the sequences of substitutions to perform if necessary.

\bigskip
If $j\in\{8,9\}$, then we replace (if needed)
\begin{equation*}
\begin{cases}
x\mapsto x+p_l,\,y\mapsto y-p_k, & \mbox{if} \quad 6|(x,y,z),\\
x\mapsto x+2p_l,\,y\mapsto y-2p_k, & \mbox{if} \quad 3\mid (x,y,z),\, 2\nmid(x,y,z),\\
x\mapsto x+3p_l,\,y\mapsto y-3p_k, & \mbox{if} \quad 2\mid (x,y,z),\, 3\nmid(x,y,z).
\end{cases}
\end{equation*}
to obtain $(x,y,6)=1$. Let $J_j=\{p:p\le p_{j-5}:p\nmid(x,y,z),\,p\in\mathbb{P}\}$. Note that $\{p:p\le p_{j-5}:\,p\in\mathbb{P}\}\supseteq\{2,3,5\}.$\\
If $(x,y,z)>1,$ then we can substitute
\begin{equation*}
x\mapsto x+Cp_l\prod\limits_{p\in J_j}p,\,y\mapsto y-Cp_k\prod\limits_{p\in J_j}p,
\end{equation*}
where $C\in\{1,2,3,4,6\}$ such that $(p_{j-4}\cdot\ldots\cdot p_{j-1},x+Cp_l\prod\limits_{p\in J_j}p)=1.$

\bigskip
If $j\in\{6,7\}$, then we replace (if needed)
\begin{equation*}
\begin{cases}
x\mapsto x+p_l,\,y\mapsto y-p_k, & \mbox{if} \quad 6|(x,y,z),\\
x\mapsto x+2p_l,\,y\mapsto y-2p_k, & \mbox{if} \quad 3\mid (x,y,z),\, 2\nmid(x,y,z),\\
x\mapsto x+3p_l,\,y\mapsto y-3p_k, & \mbox{if} \quad 2\mid (x,y,z),\, 3\nmid(x,y,z).
\end{cases}
\end{equation*}
to obtain $(x,y,6)=1$. Let $J_j=\{p:p\le p_{j-4}:p\nmid(x,y,z),\,p\in\mathbb{P}\}$. Note that $\{p:p\le p_{j-4}:\,p\in\mathbb{P}\}\supseteq\{2,3\}.$\\
If $(x,y,z)>1,$ then we can substitute
\begin{equation*}
x\mapsto x+Cp_l\prod\limits_{p\in J_j}p,\,y\mapsto y-Cp_k\prod\limits_{p\in J_j}p,
\end{equation*}
where $C\in\{1,2,3,4\}$ such that $(p_{j-3}p_{j-2}p_{j-1},x+Cp_l\prod\limits_{p\in J_j}p)=1.$

\bigskip
If $j=5$, then we replace (if needed)
\begin{equation*}
\begin{cases}
x\mapsto x+p_l,\,y\mapsto y-p_k, & \mbox{if} \quad 6|(x,y,z),\\
x\mapsto x+2p_l,\,y\mapsto y-2p_k, & \mbox{if} \quad 3\mid (x,y,z),\, 2\nmid(x,y,z),\\
x\mapsto x+3p_l,\,y\mapsto y-3p_k, & \mbox{if} \quad 2\mid (x,y,z),\, 3\nmid(x,y,z).
\end{cases}
\end{equation*}
to obtain $(x,y,6)=1$. If $(x,y,z)>1,$ then we can substitute
\begin{equation*}
x\mapsto x+6Cp_l,\,y\mapsto y-6Cp_k,
\end{equation*}
where $C\in\{1,2,3\}$ such that $(5\cdot 7,x+6Cp_l)=1$ (note that in this case $p_{j-2}=5$ and $p_{j-1}=7$).

\bigskip
If $j=4$, then we replace (if needed)
\begin{equation*}
\begin{cases}
x\mapsto x+p_l,\,y\mapsto y-p_k, & \mbox{if} \quad 6|(x,y,z),\\
x\mapsto x+2p_l,\,y\mapsto y-2p_k, & \mbox{if} \quad 3\mid (x,y,z),\, 2\nmid(x,y,z),\\
x\mapsto x+3p_l,\,y\mapsto y-3p_k, & \mbox{if} \quad 2\mid (x,y,z),\, 3\nmid(x,y,z).
\end{cases}
\end{equation*}
to obtain $(x,y,6)=1$. If $(x,y,z)>1,$ then we can substitute
\begin{equation*}
x\mapsto x+6Cp_l,\,y\mapsto y-6Cp_k,
\end{equation*}
where $C\in\{1,2\}$ such that $(5,x+6Cp_l)=1$ (note that in this case $p_{j-1}=5$).

\bigskip
If $j=3$, then we replace (if needed)
\begin{equation*}
\begin{cases}
x\mapsto x+p_l,\,y\mapsto y-p_k, & \mbox{if} \quad 6|(x,y,z),\\
x\mapsto x+2p_l,\,y\mapsto y-2p_k, & \mbox{if} \quad 3\mid (x,y,z),\, 2\nmid(x,y,z),\\
x\mapsto x+3p_l,\,y\mapsto y-3p_k, & \mbox{if} \quad 2\mid (x,y,z),\, 3\nmid(x,y,z).
\end{cases}
\end{equation*}
to obtain $(x,y,6)=1$.

\bigskip
If $j=2$, then we replace (if needed)
\begin{equation*}
x\mapsto x+p_l,\,y\mapsto y-p_k, \quad \mbox{if} \quad 2|(x,y,z)
\end{equation*}
to obtain $(x,y,2)=1$.

\bigskip
If $j=1$, then we do not have to do anything more as $0<z<p_j=2$, in other words $z=1$.
\end{proof}

\begin{lem}\label{F3}
Let $j\geq 5$. Then $2p_j>p_{j+3}$.
\end{lem}

\begin{proof}
Follows from \cite{Ram}.
\end{proof}

\begin{lem}\label{F4}
Let $s,m$ be positive integers with $s\geq 3$. Let $p_{j_1},\ldots , p_{j_s}$ be prime numbers, where $j_1>\ldots >j_s$. Then, a positive integer $n$ coprime to $p_{j_1}\cdot\ldots\cdot p_{j_s}$ can be written in the form $$\sum_{i=1}^s x_i\prod_{h\in [s]\bs\{i\}}p_{j_h}^m$$ with positive integers $x_1,\ldots ,x_s$ such that $(x_1,\ldots ,x_s)_m=1$, $(p_{j_1},x_1)=\ldots =(p_{j_s},x_s)=1$ if one of the following conditions holds:
\begin{enumerate}
\item $n\ge \left(s+1+2^{(m+1)s}\prod_{j=1, j\not\in\{j_1,\ldots ,j_s\}}^{j_1-(m+1)s-1}p_j\right)\left(\prod_{i=1}^s p_{j_i}^m\right)$ and $p_{j_1-(m+1)s}>2^{(m+1)s}$,
\item $n\ge (s+1)\left(\prod_{j=1}^{j_1}p_j\right)\left(\prod_{i=1}^{s}p_{j_i}^{m-1}\right)$.
\end{enumerate}
\end{lem}

\begin{proof}
At first, let us note that for any value of $j_s\in\N$ we can assume $x_s<p_{j_s}^m.$ If $x_s\ge p_{j_s}^m$, then we replace
\begin{equation}\label{Subs4}
x_s\mapsto x_s-\left\lfloor\frac{x_s}{p_{j_s}^m}\right\rfloor p_{j_s}^m>0,\, x_1\mapsto x_1+\left\lfloor\frac{x_s}{p_{j_s}^m}\right\rfloor p_{j_1}^m,
\end{equation}
since $(p_{j_s},x_s)=1.$

\bigskip

Consider the case (1).

We can modify the coefficients $x_1,x_2$ such that $(x_1,x_2,2)=1.$ Indeed, we replace
\begin{equation}\label{Subs5}
x_1\mapsto x_1+p_{j_1}^m,\,x_2\mapsto x_2-p_{j_2}^m, \mbox{ if } 2|(x_1,\ldots ,x_s).
\end{equation}
Let $J=\{p:p\le p_{j_1-(m+1)s-1},p\not\in\{p_{j_2},\ldots ,p_{j_s}\},p\nmid(x_1,\ldots ,x_s),\,p\in\mathbb{P}\}$. Note that $2\in J$.\\
If $(x_1,\ldots ,x_s)_m>1,$ then we can replace
\begin{equation}\label{Subs6}
x_1\mapsto x_1+Cp_{j_1}^m\prod\limits_{p\in J}p,\,x_2\mapsto x_2-Cp_{j_2}^m\prod\limits_{p\in J}p,
\end{equation}
where $C\in\{1,2,4,\ldots ,2^{(m+1)s}\}$ such that $(p_{j_1-(m+1)s}\cdot\ldots\cdot p_{j_1-1},x_1+Cp_{j_1}^m\prod\limits_{p\in J}p)=1.$\\
The last substitution can be done because:
\begin{itemize}
\item if $C,C'\in \{1,2,4,\ldots ,2^{(m+1)s}\},\, C\neq C'$, then $p_i>2^{(m+1)s}-1\ge |C-C'|$ for $j_1-(m+1)s\le i\le j_1-1$; thus $C\not\equiv C' \pmod {p_i}$;
\item for each $j_1-(m+1)s\le i\le j_1-1$ there exists at most one $C\in\{1,2,4,\ldots ,2^{(m+1)s}\}$ such that $p_i\mid x_1+Cp_{j_1}^m\prod\limits_{p\in J}p$;
\item therefore, there exists at least one $C\in\{1,2,4,\ldots ,2^{(m+1)s}\}$ such that $p_i\nmid x_1+Cp_{j_1}^m\prod\limits_{p\in J}p$ for any $j_1-(m+1)s\le i\le j_1-1$.
\end{itemize}
Now we will explain that after all these modifications we have $\left(x_1,\ldots ,x_s\right)_m=1$.
\begin{itemize}
\item Since $x_s<p_{j_s}^m$, all the $m$-th power divisors of $x_s$ are less than $p_{j_s}^m$ and thus all the $m$-th power divisors of $(x_1,\ldots ,x_s)$ are less than $p_{j_s}^m$.
\item The substitution (\ref{Subs5}) ensures us that $2$ is not a divisor of $(x_1,\ldots ,x_s)$.
\item As $2\in J$, after substitution (\ref{Subs6}) the number $2$ is still not a divisor of $(x_1,\ldots ,x_s)$.
\item If $p\in J$, then substitution (\ref{Subs6}) does not change divisibility of $x_1$ and $x_2$ by $p$. Hence $p\nmid (x_1,\ldots ,x_s)$.
\item If $p\not\in J$ and $p<p_{j_1-(m+1)s}$, then $p\mid (x_1,\ldots ,x_s)$ before substitution (\ref{Subs6}). Because $p\nmid Cp_{j_1}p_{j_2}\prod\limits_{q\in J}q$, the substitution (\ref{Subs6}) changes the divisiblity of $x_1$ and $x_2$ by $p$, i.e. $p\nmid x_1$ and $p\nmid x_2$ after substitution (\ref{Subs6}).
\item If $p=p_i,$ where $j_1-(m+1)s\leq i\leq j_1-1$, then $p_i\nmid x_1$ after substitution (\ref{Subs6}). This follows directly from the choice of $C$.
\end{itemize}
If the beginning value of $x_2$ is greater than $p_{j_2}^m\left(1+2^{(m+1)s}\prod_{j=1, j\not\in\{j_1,\ldots ,j_s\}}^{j_1-(m+1)s-1}p_j\right)$, then it remains positive after substitutions (\ref{Subs5}) and (\ref{Subs6}). Hence, if $n-p_{j_1}^mp_{j_2}^m\cdot\ldots\cdot p_{j_s}^m\left(1+2^{(m+1)s}\prod_{j=1, j\not\in\{j_1,\ldots ,j_s\}}^{j_1-(m+1)s-1}p_j\right)\ge sp_{j_1}^mp_{j_2}^m\cdot\ldots\cdot p_{j_s}^m$, then by Lemma \ref{F1} b) the number $n-p_{j_1}^mp_{j_2}^m\cdot\ldots\cdot p_{j_s}^m\left(1+2^{(m+1)s}\prod_{j=1, j\not\in\{j_1,\ldots ,j_s\}}^{j_1-(m+1)s-1}p_j\right)$ can be written in the form $\sum_{i=1}^s \tilde{x}_i\prod_{h\in [s]\bs\{i\}}p_{j_h}^m$ for some $\tilde{x}_1,\ldots ,\tilde{x}_s\in\N$. Thus, $n$ can be written in the form $\sum_{i=1}^s x_i\prod_{h\in [s]\bs\{i\}}p_{j_h}^m$ for some $x_1,x_3,\ldots ,x_s>0$ and $x_2>p_{j_2}^m\left(1+2^{(m+1)s}\prod_{j=1, j\not\in\{j_1,\ldots ,j_s\}}^{j_1-(m+1)s-1}p_j\right)$. Hence, applying substitutions (\ref{Subs4})--(\ref{Subs6}), we conclude that $n$ can be written as $\sum_{i=1}^s x_i\prod_{h\in [s]\bs\{i\}}p_{j_h}^m$ for some $x_1,\ldots ,x_s\geq 0$ with $(x_1,\ldots ,x_s)=1$. The equality $(p_{j_1},x_1)=\ldots =(p_{j_s},x_s)=1$ follows directly from the assumption $(n,p_{j_1}\cdot\ldots\cdot p_{j_s})=1$.

\bigskip

Now we consider the case (2).

Let $J=\{p:p\le p_{j_1-1},p\not\in\{p_{j_2},\ldots ,p_{j_s}\},p\nmid(x_1,\ldots ,x_s),\,p\in\mathbb{P}\}$. If $(x_1,\ldots ,x_s)>1,$ then we can replace
\begin{equation}\label{Subs7}
x_1\mapsto x_1+p_{j_1}^m\prod\limits_{p\in J}p,\,x_2\mapsto x_2-p_{j_2}^m\prod\limits_{p\in J}p.
\end{equation}

Now we will explain that after the above modification we have $\left(x_1,\ldots ,x_s\right)_m=1$.
\begin{itemize}
\item Since $x_s<p_{j_s}^m$, all the $m$-th power divisors of $x_s$ are less than $p_{j_s}^m$ and thus all the prime divisors of $(x_1,\ldots ,x_s)$ are less than $p_{j_s}^m$.
\item If $p\in J$, then substitution (\ref{Subs7}) does not change divisibility of $x_1$ and $x_2$ by $p$. Hence $p\nmid (x_1,\ldots ,x_s)$.
\item If $p\not\in J$, then $p\mid (x_1,\ldots ,x_s)$ before substitution (\ref{Subs7}). Because $p\nmid p_{j_1}p_{j_2}\prod\limits_{q\in J}q$, the substitution (\ref{Subs7}) changes the divisiblity of $x_1$ and $x_2$ by $p$, i.e. $p\nmid x_1$ and $p\nmid x_2$ after substitution (\ref{Subs7}).
\end{itemize}
If the beginning value of $x_2$ is greater than $p_{j_2}^m(\prod_{j=1, j\not\in\{j_s,\ldots ,j_2\}}^{j_1-1}p_j)$, then it remains positive after substitution (\ref{Subs7}). Hence, if $n-p_{j_1}^mp_{j_2}^m\cdots p_{j_s}^m(\prod_{j=1, j\not\in\{j_s,\ldots ,j_2\}}^{j_1-1}p_j)=n-\left(\prod_{j=1}^{j_1}p_j\right)\left(\prod_{i=1}^{s}p_{j_i}^{m-1}\right)\ge sp_{j_1}^mp_{j_2}^m\cdots p_{j_s}^m$, then by Lemma \ref{F1} b) the number $n-\left(\prod_{j=1}^{j_1}p_j\right)\left(\prod_{i=1}^{s}p_{j_i}^{m-1}\right)$ can be written in the form $\sum_{i=1}^s \tilde{x}_i\prod_{h\in [s]\bs\{i\}}p_{j_h}^m$ for some $\tilde{x}_1,\ldots ,\tilde{x}_s\in\N$. Thus, $n$ can be written in the form $\sum_{i=1}^s x_i\prod_{h\in [s]\bs\{i\}}p_{j_h}^m$ for some $x_1,x_3,\ldots ,x_s>0$ and $x_2>p_{j_2}^m(\prod_{j=1, j\not\in\{j_s,\ldots ,j_2\}}^{j_1-1}p_j)$. Hence, applying substitution (\ref{Subs7}), we conclude that $n$ can be written as $\sum_{i=1}^s x_i\prod_{h\in [s]\bs\{i\}}p_{j_h}^m$ for some $x_1,\ldots ,x_s\geq 0$ with $(x_1,\ldots ,x_s)=1$. The equality $(p_{j_1},x_1)=\ldots =(p_{j_s},x_s)=1$ follows directly from the assumption $(n,p_{j_1}\cdots p_{j_s})=1$.\\
At last, let us note that the assumption $n\ge (s+1)\left(\prod_{j=1}^{j_1}p_j\right)\left(\prod_{i=1}^{s}p_{j_i}^{m-1}\right)$ implies $n\ge \left(\prod_{j=1}^{j_1}p_j\right)\left(\prod_{i=1}^{s}p_{j_i}^{m-1}\right)+s\prod_{i=1}^s p_{j_i}^m$.
\end{proof}

\section{Proof of Theorem \ref{Ss}}

We will show that $n=\sum\limits_{k=1}^sa_k\mu_k,$ where for $k<s$ we put
\begin{equation*}
\mu_k=x_k\prod\limits_{\{i,j\}\subset [s]\bs\{k\},\,i\neq j}p_{\{i,j\}}^m
\end{equation*}
and
\begin{equation*}
\mu_s=\prod\limits_{\{i,j\}\subset [s-1],\,i\neq j}p_{\{i,j\}}^m,
\end{equation*}
where $x_k\in\mathbb{N},\,\,\forall_{i,j\in [s-1],\,i\neq j}\,\,p_{\{i,j\}}^m\nmid (x_i,x_j)_m.$ Note that $p_{\{i,j\}}^m\nmid (x_i,x_j)_m\Leftrightarrow p_{\{i,j\}}^m\nmid (x_i,x_j)\Leftrightarrow (p_{\{i,j\}}^m,x_i,x_j)_m=1$. Using set-theoretical notation, we have $p_{\{i,j\}}=p_{\{j,i\}}$.\\
Then, one can see that:
\begin{enumerate}
\item[a)] $p_{\{i,j\}}^m\mid \text{GCD}\{\mu_k:k\in [s]\bs\{i,j\}\}$ for $i,j\in [s],\,i\neq j;$
\item[b)] $p_{\{i,j\}}^m\nmid (\mu_{i},\mu_{j})_m$ for $i,j\in [s-1],\,i\neq j;$
\item[c)] $(\mu_1,\ldots,\mu_s)_m=1.$
\end{enumerate}
Let us consider the expression
$$C(x_1^{(1)},x_2^{(1)},\ldots,x_{s-1}^{(1)})=\sum\limits_{k\in [s-1]}a_kx_k^{(1)}\prod\limits_{\{i,j\}\subset [s]\bs\{k\},\,i\neq j}p_{\{i,j\}}^m,$$
where $x_k^{(1)}\in\mathbb{Z}.$ Assume additionally that
$$x_k^{(1)}>\sum_{l\in [s-1],\,l>k}\frac{a_l}{(a_k,a_l)}\prod_{j\in [s]\bs\{k,l\}}p_{\{j,k\}}^m,$$ 
where $k\in [s-1]$. For $k\in [s-1]$ we put
\begin{equation*}
x_k^{(2)}=x_k^{(1)}+\sum_{l\in [k-1]}\epsilon_{l,k}\frac{a_l}{(a_k,a_l)}\prod_{j\in [s]\bs\{k,l\}}p_{\{j,k\}}^m-\sum_{l\in [s-1],\,l>k}\epsilon_{l,k}\frac{a_l}{(a_k,a_l)}\prod_{j\in [s]\bs\{k,l\}}p_{\{j,k\}}^m,
\end{equation*}
where $\epsilon_{l,k}=\delta(p_{\{l,k\}}^m\mid (x_l^{(1)},x_k^{(1)})_m)$. 
Note that $\epsilon_{l,k}=\epsilon_{k,l}$.
We see that
\begin{align*}
&C(x_1^{(2)},x_2^{(2)},\ldots,x_{s-1}^{(2)})=\sum\limits_{k\in [s-1]}a_kx_k^{(2)}\prod\limits_{\{i,j\}\subset [s]\bs\{k\},\,i\neq j}p_{\{i,j\}}^m\\
&=\sum\limits_{k\in [s-1]}a_k\left(x_k^{(1)}+\sum_{l\in [k-1]}\epsilon_{l,k}\frac{a_l}{(a_k,a_l)}\prod_{j\in [s]\bs\{k,l\}}p_{\{j,k\}}^m-\sum_{l\in [s-1],\,l>k}\epsilon_{l,k}\frac{a_l}{(a_k,a_l)}\prod_{j\in [s]\bs\{k,l\}}p_{\{j,k\}}^m\right)\\
&\quad\cdot\prod\limits_{\{i,j\}\subset [s]\bs\{k\},\,i\neq j}p_{\{i,j\}}^m\\
&=\sum\limits_{k\in [s-1]}a_kx_k^{(1)}\prod\limits_{\{i,j\}\subset [s]\bs\{k\},\,i\neq j}p_{\{i,j\}}^m+\sum\limits_{k\in [s-1]}a_k\sum_{l\in [k-1]}\epsilon_{l,k}\frac{a_l}{(a_k,a_l)}\cdot\frac{1}{p_{\{l,k\}}^m}\prod\limits_{\{i,j\}\subset [s],\,i\neq j}p_{\{i,j\}}^m\\
&\quad -\sum\limits_{k\in [s-1]}a_k\sum_{l\in [s-1],\,l>k}\epsilon_{l,k}\frac{a_l}{(a_k,a_l)}\cdot\frac{1}{p_{\{l,k\}}^m}\prod\limits_{\{i,j\}\subset [s],\,i\neq j}p_{\{i,j\}}^m\\
&=C(x_1^{(1)},x_2^{(1)},\ldots,x_{s-1}^{(1)})+\sum\limits_{k\in [s-1]}\sum_{l\in [k-1]}\epsilon_{l,k}\frac{a_ka_l}{(a_k,a_l)}\cdot\frac{1}{p_{\{l,k\}}^m}\prod\limits_{\{i,j\}\subset [s],\,i\neq j}p_{\{i,j\}}^m\\
&\quad -\sum\limits_{l\in [s-1]}\sum_{k\in [l-1]}\epsilon_{l,k}\frac{a_ka_l}{(a_k,a_l)}\cdot\frac{1}{p_{\{l,k\}}^m}\prod\limits_{\{i,j\}\subset [s],\,i\neq j}p_{\{i,j\}}^m\\
&=C(x_1^{(1)},x_2^{(1)},\ldots,x_{s-1}^{(1)}),
\end{align*}
where we used the equality $\epsilon_{l,k}=\epsilon_{k,l}$.
We proved that if a positive integer can be written in the form
\begin{equation}\label{pewnie}
C(x_1^{(1)},x_2^{(1)},\ldots,x_{s-1}^{(1)})+a_s\prod\limits_{\{i,j\}\subset [s-1],\,i\neq j}p_{\{i,j\}}^m,
\end{equation}
where $x_k^{(1)}\in\mathbb{N},\,\,$ $x_k^{(1)}>\sum_{l\in [s-1],\,l>k}\frac{a_l}{(a_k,a_l)}\prod_{j\in [s]\bs\{k,l\}}p_{\{j,k\}}^m$, $k\in [s-1]$, then it can be written in the form 
$$C(x_1^{(2)},x_2^{(2)},\ldots,x_{s-1}^{(2)})+a_s\prod\limits_{\{i,j\}\subset [s-1],\,i\neq j}p_{\{i,j\}}^m,$$ where  
$x_k^{(2)}\in\mathbb{N},\,\,\forall_{j\in [s-1]\bs\{k\}}\,\,p_{\{j,k\}}^m\nmid (x_j^{(2)},x_k^{(2)})_m.$\\
It remains to give a value $n_0$ such that  each positive integer $n\ge n_0$ can be written as (\ref{pewnie})
where $x_k^{(1)}\in\mathbb{N},\,\,$ $x_k>\sum_{l\in [s-1],\,l>k}\frac{a_l}{(a_k,a_l)}\prod_{j\in [s]\bs\{k,l\}}p_{\{j,k\}}^m.$

We put $x_k^{(1)}=x_k^{(3)}+\sum_{l\in [s-1],\,l>k}\frac{a_l}{(a_k,a_l)}\prod_{j\in [s]\bs\{k,l\}}p_{\{j,k\}}^m,$ where $x_k^{(3)}>0,\,k\in [s-1]$.\\
It remains to give a value $n_0$ such that  each positive integer $n\ge n_0$ can be written as 
$$C(x_1^{(3)},x_2^{(3)},\ldots,x_{s-1}^{(3)})+a_s\prod\limits_{\{i,j\}\subset [s-1],\,i\neq j}p_{\{i,j\}}^m+\sum_{k=1}^{s-2}\left(a_k\prod_{\{i,j\}\subset [s]\bs\{k\},\,i\neq j}p_{\{i,j\}}^m\right)\cdot\left(\sum_{l\in [s-1],\,l>k}\frac{a_l}{(a_k,a_l)}\prod_{j\in [s]\bs\{k,l\}}p_{\{j,k\}}^m\right),$$ where $x_k^{(3)}\in\mathbb{N},\,k\in [s-1]$, i.e.:
\begin{align*}
C(x_1^{(3)},x_2^{(3)},\ldots,x_{s-1}^{(3)})+a_s\prod\limits_{\{i,j\}\subset [s-1],\,i\neq j}p_{\{i,j\}}^m+\sum_{k=1}^{s-2}\left(\prod_{\{i,j\}\subset [s],\,i\neq j}p_{\{i,j\}}^m\right)\cdot\left(\sum_{l\in [s-1],\,l>k}\frac{a_ka_l}{(a_k,a_l)}\cdot\frac{1}{p_{\{l,k\}}^m}\right).
\end{align*}

By Lemma \ref{F1}, b) we may take:
$$n_0\ge \sum\limits_{k\in [s-1]}\frac{b_{s-1} b_k}{(b_{s-1},b_k)}+a_s\prod\limits_{\{i,j\}\subset [s-1],\,i\neq j}p_{\{i,j\}}^m+\sum_{k=1}^{s-2}\left(\prod_{\{i,j\}\subset [s],\,i\neq j}p_{\{i,j\}}^m\right)\cdot\left(\sum_{l\in [s-1],\,l>k}\frac{a_ka_l}{(a_k,a_l)}\cdot\frac{1}{p_{\{l,k\}}^m}\right),$$
where $b_k=a_k\prod_{\{i,j\}\subset [s]\bs\{k\},\,i\neq j}p_{\{i,j\}}^m.$ We can apply Lemma \ref{F1} to $b_k$, $k\in [s-1]$, as $(b_1,b_2,\ldots,b_{s-1})=1$ follows from the assumptions $p_{\{i,j\}}\nmid (a_i,a_j)$ for each $i,j\in [s-1],\,\,i\neq j,$ and $p_{\{i,s\}}\nmid a_i$ for each $i\in [s-1]$.
But $\frac{b_{s-1}b_k}{(b_{s-1},b_k)}\le\frac{a_{s-1} a_k}{(a_{s-1},a_k)}\cdot\frac{\prod\limits_{\{i,j\}\subset [s],\, i\neq j}p_{\{i,j\}}^m}{p_{\{k,s-1\}}^m}$ for $k\in [s-2]$, hence we may put
\begin{align*}
n_0=&\sum\limits_{k\in [s-2]}\frac{a_{s-1} a_k}{(a_{s-1},a_k)}\cdot\frac{\prod\limits_{\{i,j\}\subset [s],\, i\neq j}p_{\{i,j\}}^m}{p_{\{k,s-1\}}^m}+a_{s-1}\prod\limits_{\{i,j\}\subset [s]\bs\{s-1\},\,i\neq j}p_{\{i,j\}}^m+a_s\prod\limits_{\{i,j\}\subset [s-1],\,i\neq j}p_{\{i,j\}}^m\\
&+\sum_{k=1}^{s-2}\left(\prod_{\{i,j\}\subset [s],\,i\neq j}p_{\{i,j\}}^m\right)\cdot\left(\sum_{l\in [s-1],\,l>k}\frac{a_ka_l}{(a_k,a_l)}\cdot\frac{1}{p_{\{l,k\}}^m}\right)\\
=&a_{s-1}\prod\limits_{\{i,j\}\subset [s]\bs\{s-1\},\,i\neq j}p_{\{i,j\}}^m+a_s\prod\limits_{\{i,j\}\subset [s-1],\,i\neq j}p_{\{i,j\}}^m\\
&+\left(\prod_{\{i,j\}\subset [s],\,i\neq j}p_{\{i,j\}}^m\right)\sum_{k=1}^{s-2}\left(\frac{a_{s-1}a_k}{(a_{s-1},a_k)}\cdot\frac{2}{p_{\{s-1,k\}}^m}+\sum_{l\in [s-2],\,l>k}\frac{a_ka_l}{(a_k,a_l)}\cdot\frac{1}{p_{\{l,k\}}^m}\right).
\end{align*}
This means that $n\ge n_0$ can be represented as $n=\sum\limits_{k=1}^sa_k\mu_k,$ where
\begin{equation*}
\mu_k=x_k\prod\limits_{\{i,j\}\subset [s]\bs\{k\},\,i\neq j}p_{\{i,j\}}^m,
\end{equation*}
\begin{equation*}
\mu_s=\prod\limits_{\{i,j\}\subset [s-1],\,i\neq j}p_{\{i,j\}}^m
\end{equation*}
and $x_k\in\mathbb{N},\,\,\forall_{i,j\in [s-1],\,i\neq j}\,\,p_{\{i,j\}}^m\nmid (x_i,x_j).$ In particular, $n\ge n_0$ belongs to $S_{s,t}^m(a_1,a_2,\ldots,a_s)$ for $t\in [s-2]$ as $S_{s,t}^m(a_1,a_2,\ldots,a_s)\supseteq S_{s,s-2}^m(a_1,a_2,\ldots,a_s)$. This ends the proof.

\section{Proof of Theorem \ref{S3}}
Let
\begin{align*}
j:=&\min\{i\in\N: p_i\nmid n\},\\ k:=&\min\{i\in\N: p_i\nmid n\}\bs\{j\},\\ l:=&\min\{i\in\N: p_i\nmid n\}\bs\{j,k\}
\end{align*}
and (recall that $p_j=1$ for $j\leq 0$)
\begin{align*}
l':=&\max\{i\in\Z: p_i\mid n\},\\ k':=&\max\{i\in\Z: p_i\mid n\}\bs\{l'\},\\ j':=&\max\{i\in\Z: p_i\mid n\}\bs\{l',k'\}.
\end{align*}
By a careful analysis case by case we can see that $l'\geq l-3$, $k'\geq k-3$ and $j'\geq j-3$.\\
Moreover, $n\geq\prod_{i=1, i\not\in\{j,k,l\}}^{l'}p_i$.

If $j\geq 10$, then $l'>k'>j'\geq j-3\geq 7$ and by Lemma \ref{F3} we have $2p_{j'}>p_j$, $2p_{k'}>p_k$ and $2p_{l'}>p_l$. Hence,
\begin{align*}
n &\ge \prod_{i=1, i\not\in\{j,k,l\}}^{l'}p_i=\prod_{i=1}^{j-6}p_i\cdot p_{j-5}p_{j-4}\cdot\prod_{i=j-3, i\not\in\{j,k,l\}}^{l'}p_i\ge\prod_{i=1}^{j-6}p_i\cdot p_5p_6p_{j'}p_{k'}p_{l'}=\prod_{i=1}^{j-6}p_i\cdot 143p_{j'}p_{k'}p_{l'}\\
&>9\prod_{i=1}^{j-6}p_i\cdot 2p_{j'}\cdot 2p_{k'}\cdot 2p_{l'}>9\prod_{i=1}^{j-6}p_i\cdot p_jp_kp_l>\left(8\prod_{i=1}^{j-6}p_i+6\right)p_jp_kp_l.
\end{align*}
By Lemma \ref{F2}, the number $n$ can be written in the form $p_jp_kx+p_jp_ly+p_kp_lz$, where $x,y,z\in\N$ are such that $(x,y,z)=1$.

Let $j\in\{8,9\}$ and $l-j\geq 3$. Then $l'\geq l-2$, $k'\geq k-3$ and $j'\geq j-2\geq 6$. By Lemma \ref{F3} we have $2p_{j'}>p_j$, $2p_{k'}>p_k$ and $2p_{l'}>p_l$. Hence,
\begin{align*}
n & \geq\prod_{i=1}^{j-5}p_i\cdot p_{j-4}p_{j-3}\cdot\prod_{i=j-2, i\not\in\{j,k,l\}}^{l'}p_i\geq\prod_{i=1}^{j-5}p_i\cdot p_4p_5p_{j'}p_{k'}p_{l'}=\prod_{i=1}^{j-5}p_i\cdot 77p_{j'}p_{k'}p_{l'}>7\prod_{i=1}^{j-5}p_i\cdot 2p_{j'}\cdot 2p_{k'}\cdot 2p_{l'}\\
& >7\prod_{i=1}^{j-5}p_i\cdot p_jp_kp_l>\left(6\prod_{i=1}^{j-5}p_i+6\right)p_jp_kp_l.
\end{align*}
By Lemma \ref{F2}, the number $n$ can be written in the form $p_jp_kx+p_jp_ly+p_kp_lz$, where $x,y,z\in\N$ are such that $(x,y,z)=1$.

If $j\in\{8,9\}$ and $l-j=2$, i. e. $k=j+1\leq 10$ and $l=j+2\leq 11$, then by Lemma \ref{F2}, each integer $$n\geq\left(6\prod_{i=1}^4p_i+6\right)p_9p_{10}p_{11}=26 177 082\geq\left(6\prod_{i=1}^{j-5}p_i+6\right)p_jp_{j+1}p_{j+2}$$
can be written in the form $p_jp_kx+p_jp_ly+p_kp_lz$, where $x,y,z\in\N$ are such that $(x,y,z)=1$.

Let $j\in\{6,7\}$ and $l-j\geq 4$. Then $l'\geq l-2$, $k'\geq k-2$ and $j'\geq j-1\geq 5$. By Lemma \ref{F3} we have $2p_{k'}>p_k$ and $2p_{l'}>p_l$. Moreover, we easily check that $\frac{7}{4}p_{j'}\geq \frac{7}{4}p_{j-1}>p_j$. Hence,
\begin{align*}
n & \geq\prod_{i=1}^{j-4}p_i\cdot p_{j-3}p_{j-2}\cdot\prod_{i=j-1, i\not\in\{j,k,l\}}^{l'}p_i\geq\prod_{i=1}^{j-4}p_i\cdot p_3p_4p_{j'}p_{k'}p_{l'}=\prod_{i=1}^{j-4}p_i\cdot 35p_{j'}p_{k'}p_{l'}=5\prod_{i=1}^{j-4}p_i\cdot \frac{7}{4}p_{j'}\cdot 2p_{k'}\cdot 2p_{l'}\\
& >5\prod_{i=1}^{j-4}p_i\cdot p_jp_kp_l\geq\left(4\prod_{i=1}^{j-5}p_i+6\right)p_jp_kp_l.
\end{align*}
By Lemma \ref{F2}, the number $n$ can be written in the form $p_jp_kx+p_jp_ly+p_kp_lz$, where $x,y,z\in\N$ are such that $(x,y,z)=1$.

If $j\in\{6,7\}$ and $l-j\leq 3$, i. e. $k\leq j+2\leq 9$ and $l\leq j+3\leq 10$, then by Lemma \ref{F2}, each integer $$n\geq\left(4\prod_{i=1}^3p_i+6\right)p_7p_9p_{10}=1428714\geq\left(4\prod_{i=1}^{j-4}p_i+6\right)p_jp_kp_l$$
can be written in the form $p_jp_kx+p_jp_ly+p_kp_lz$, where $x,y,z\in\N$ are such that $(x,y,z)=1$.

Let $j=5$ and $l-j\geq 5$. Then $l'\geq l-2$, $k'\geq l-3\geq j+2=7$ and $j'>j$. By Lemma \ref{F3} we have $2p_{k'}>p_k$ and $2p_{l'}>p_l$. Hence,
\begin{align*}
n & \geq p_1p_2p_3p_4\cdot\prod_{i=6, i\not\in\{k,l\}}^{l'}p_i\geq p_1p_2\cdot 35\cdot p_{j'}p_{k'}p_{l'}>4p_1p_2\cdot p_j\cdot 2p_{k'}\cdot 2p_{l'}>4p_1p_2\cdot p_jp_kp_l=(3p_1p_2+6)p_jp_kp_l.
\end{align*}
By Lemma \ref{F2}, the number $n$ can be written in the form $p_jp_kx+p_jp_ly+p_kp_lz$, where $x,y,z\in\N$ are such that $(x,y,z)=1$.

If $j=5$ and $l-j\leq 4$, i. e. $k\leq j+3=8$ and $l\leq j+4=9$, then by Lemma \ref{F2}, each integer $$n\geq\left(3p_1p_2+6\right)p_5p_8p_9=115368\geq\left(3p_1p_2+6\right)p_jp_kp_l$$
can be written in the form $p_jp_kx+p_jp_ly+p_kp_lz$, where $x,y,z\in\N$ are such that $(x,y,z)=1$.

Let $j=4$ and $l-j\geq 6$. Then $l'\geq l-2$, $k'\geq l-3\geq j+3=7$ and $j'\geq j+2=6$. By Lemma \ref{F3} we have $2p_{k'}>p_k$ and $2p_{l'}>p_l$. Hence,
\begin{align*}
n & \geq p_1p_2p_3p_5\cdot\prod_{i=6, i\not\in\{k,l\}}^{l'}p_i\geq p_1p_2\cdot 55\cdot p_{j'}p_{k'}p_{l'}>3p_1p_2\cdot p_j\cdot 2p_{k'}\cdot 2p_{l'}>3p_1p_2\cdot p_jp_kp_l=(2p_1p_2+6)p_jp_kp_l.
\end{align*}
By Lemma \ref{F2}, the number $n$ can be written in the form $p_jp_kx+p_jp_ly+p_kp_lz$, where $x,y,z\in\N$ are such that $(x,y,z)=1$.

If $j=4$ and $l-j\leq 5$, i. e. $k\leq j+4=8$ and $l\leq j+5=9$, then by Lemma \ref{F2}, each integer $$n\geq\left(2p_1p_2+6\right)p_4p_8p_9=55 062\geq\left(2p_1p_2+6\right)p_jp_kp_l$$
can be written in the form $p_jp_kx+p_jp_ly+p_kp_lz$, where $x,y,z\in\N$ are such that $(x,y,z)=1$.

Let $j=3$ and $l-j\geq 6$. Then $l'\geq l-2$, $k'\geq l-3\geq j+3=6$ and $j'\geq j+2=5$. By Lemma \ref{F3} we have $2p_{k'}>p_k$ and $2p_{l'}>p_l$. Hence,
\begin{align*}
n & \geq p_1p_2p_4\cdot\prod_{i=5, i\not\in\{k,l\}}^{l'}p_i\geq 6\cdot 7\cdot p_{j'}p_{k'}p_{l'}>6p_j\cdot 2p_{k'}\cdot 2p_{l'}>6p_jp_kp_l.
\end{align*}
By Lemma \ref{F2}, the number $n$ can be written in the form $p_jp_kx+p_jp_ly+p_kp_lz$, where $x,y,z\in\N$ are such that $(x,y,z)=1$.

If $j=3$ and $l-j\leq 5$, i. e. $k\leq j+4=7$ and $l\leq j+5=8$, then by Lemma \ref{F2}, each integer $$n\geq 6p_3p_7p_8=9690\geq 6p_jp_kp_l$$
can be written in the form $p_jp_kx+p_jp_ly+p_kp_lz$, where $x,y,z\in\N$ are such that $(x,y,z)=1$.

Let $j=2$ and $l-j\geq 7$. Then $l'\geq l-2$, $k'\geq l-3\geq j+4=6$ and $j'\geq j+3=5$. By Lemma \ref{F3} we have $2p_{k'}>p_k$ and $2p_{l'}>p_l$. Hence,
\begin{align*}
n & \geq p_1p_3p_4\cdot\prod_{i=5, i\not\in\{k,l\}}^{l'}p_i\geq 70p_{j'}p_{k'}p_{l'}>4p_j\cdot 2p_{k'}\cdot 2p_{l'}>4p_jp_kp_l.
\end{align*}
By Lemma \ref{F2}, the number $n$ can be written in the form $p_jp_kx+p_jp_ly+p_kp_lz$, where $x,y,z\in\N$ are such that $(x,y,z)=1$.

If $j=2$ and $l-j\leq 6$, i. e. $k\leq j+5=7$ and $l\leq j+6=8$, then by Lemma \ref{F2}, each integer $$n\geq 4p_2p_7p_8=3876\geq 4p_jp_kp_l$$
can be written in the form $p_jp_kx+p_jp_ly+p_kp_lz$, where $x,y,z\in\N$ are such that $(x,y,z)=1$.

Let $j=1$ and $l-j\geq 7$. Then $l'\geq l-2$, $k'\geq k-3\geq j+4=5$ and $j'\geq j+3=4$. By Lemma \ref{F3} we have $2p_{k'}>p_k$ and $2p_{l'}>p_l$. Hence,
\begin{align*}
n & \geq p_2p_3\cdot\prod_{i=4, i\not\in\{k,l\}}^{l'}p_i\geq 3\cdot 5\cdot p_{j'}p_{k'}p_{l'}>3p_s\cdot 2p_{k'}\cdot 2p_{l'}>3p_jp_kp_l.
\end{align*}
By Lemma \ref{F2}, the number $n$ can be written in the form $p_jp_kx+p_jp_ly+p_kp_lz$, where $x,y,z\in\N$ are such that $(x,y,z)=1$.

If $j=1$ and $l-j\leq 6$, i. e. $k\leq j+5=6$ and $l\leq j+6=7$, then by Lemma \ref{F2}, each integer $$n\geq 3p_1p_6p_7=1326\geq 3p_jp_kp_l$$
can be written in the form $p_jp_kx+p_jp_ly+p_kp_lz$, where $x,y,z\in\N$ are such that $(x,y,z)=1$.

Summing up, if $j\in\N$ is arbitrary and $n\geq 26 177 082$, then $n=p_jp_kx+p_jp_ly+p_kp_lz$ for some $x,y,z\in\N$ such that $(x,y,z)=1$. Thus $n\in\cal{S}_3$.

\section{Proof of Theorem \ref{Sss-1}}

Let
\begin{align*}
j_s:=&\min\{j\in\N: p_j\nmid n\},\\ j_i:=&\min\{j\in\N: p_j\nmid n\}\bs\{j_s,\ldots ,j_{i+1}\},\ 1\leq i\leq s-1
\end{align*}
and (recall that $p_j=1$ for $j\leq 0$)
\begin{align*}
j'_1:=&\max\{j\in\Z: p_j\mid n\},\\ j'_i:=&\max\{j\in\Z: p_j\mid n\}\bs\{j'_1,\ldots ,j'_{i-1}\},\ 2\leq i\leq ms+1.
\end{align*}
We can see that $j'_i\geq j_1-s-i+1$ for $i\in\{1,\ldots ,ms+1\}$ and $j_i\leq j_1-i+1$ for $i\in\{1,\ldots ,s\}$. In particular, $j'_{us+i}\geq j_i-(u+1)s$ for $i\in\{1,\ldots ,s\}$ and $u\in\{0,\ldots ,m-1\}$. Moreover, $$n\geq\left(\prod_{j=1, j\not\in\{j_1,\ldots ,j_s\}}^{j'_{ms+1}}p_j\right)\left(\prod_{i=1}^{ms} p_{j'_i}\right).$$

First, let us consider the case when $j_1\geq c_{s,m}$. Then,
\begin{align*}
n &\ge \left(\prod_{j=1, j\not\in\{j_s,\ldots ,j_1\}}^{j'_{ms+1}-1}p_j\right)\cdot p_{j'_{ms+1}}\cdot \left(\prod_{i=1}^{ms} p_{j'_i}\right)\ge \left(\prod_{j=1, j\not\in\{j_s,\ldots ,j_1\}}^{j'_{ms+1}-1}p_j\right)\cdot p_{j_1+(m+1)s}\cdot \left(\prod_{i=1}^{ms} p_{j'_i}\right)\\
& >\left(\prod_{j=1, j\not\in\{j_s,\ldots ,j_1\}}^{j_1-(m+1)s-1}p_j\right)\cdot 2^{(2m+1)s+1}\cdot \left(\prod_{i=1}^s\prod_{u=1}^m p_{j_i-us}\right)> \left(2^{(m+1)s+1}\prod_{j=1, j\not\in\{j_s,\ldots ,j_1\}}^{j_1-(m+1)s-1}p_j\right)\left(\prod_{i=1}^s\prod_{u=1}^m 2p_{j_i-us}\right)\\
& >\left(s+1+2^{(m+1)s}\prod_{j=1, j\not\in\{j_s,\ldots ,j_1\}}^{j_1-(m+1)s-1}p_j\right)\left(\prod_{i=1}^s p_{j_i}^m\right).
\end{align*}
By Lemma \ref{F4}, the number $n$ can be written in the form $\sum_{i=1}^s x_i\prod_{h\in [s]\bs\{i\}}p_{j_h}^m$ with positive integers $x_1,\ldots ,x_s$ such that $(x_1,\ldots ,x_s)_m=1$, $(p_{j_1},x_1)=\ldots =(p_{j_s},x_s)=1$.

We are left with the case of $j_1<c_{s,m}$. Let $n\ge (s+1)\left(\prod_{j=1}^{c_{s,m}-1}p_j\right)\left(\prod_{i=1}^{s}p_{c_{s,m}-i}^{m-1}\right)$. Since
\begin{align*}
n &\ge (s+1)\left(\prod_{j=1}^{c_{s,m}-1}p_j\right)\left(\prod_{i=1}^{s}p_{c_{s,m}-i}^{m-1}\right)\ge (s+1)\left(\prod_{j=1}^{j_1}p_j\right)\left(\prod_{i=1}^{s}p_{j_i}^{m-1}\right),
\end{align*}
by Lemma \ref{F4} the number $n$ can be written in the form $\sum_{i=1}^s x_i\prod_{h\in [s]\bs\{i\}}p_{j_h}^m$ with positive integers $x_1,\ldots ,x_s$ such that $(x_1,\ldots ,x_s)_m=1$, $(p_{j_1},x_1)=\ldots =(p_{j_s},x_s)=1$.

\section{Further study}

We are left with some open problems. The first one is whether the assumption $(a_1,a_2,\ldots,a_{s-1})=1$ can be relaxed to $(a_1,a_2,\ldots,a_s)=1$. We suppose that the answer on this question is affirmative.

\begin{conj}
Let $s\ge 3.$ Let $a_k$, where $k\in [s]$, be positive integers such that $(a_1,a_2,\ldots,a_s)=1$. Then there exists a computable constant $c=c(a_1,a_2,\ldots,a_s)$ such that every positive integer at least equal to $c$ belongs to the set $S_{s,s-2}^m(a_1,a_2,\ldots,a_s)$.
\end{conj}

Another problem concerns stems of the form $S_{s,s-1}^m(a_1,a_2,\ldots,a_s)$, where $s\ge 3$.

\begin{ques}
Let $s\ge 3.$ Let $a_k$, where $k\in [s]$, be positive integers such that $(a_1,a_2,\ldots,a_s)=1$. Is it true that every sufficiently large positive integer belongs to the set $S_{s,s-1}^m(a_1,a_2,\ldots,a_s)$? If yes, what is the constant $c=c(a_1,a_2,\ldots,a_s)$ such that every positive integer at least equal to $c$ belongs to the set $S_{s,s-1}^m(a_1,a_2,\ldots,a_s)$?
\end{ques}

The next step in the study in the area of Frobenius problem with restrictions could be considering the following more general sets:
\begin{align*}
\cal{T}_{s,k}^m(a_1,a_2,\ldots a_s)= & \left\{\sum_{i=1}^sa_i\mu_i: \mu_1,\ldots ,\mu_s\in\N, (\mu_{i_1},\ldots ,\mu_{i_k})_m>1\mbox{ for any }1\leq i_1<\ldots <i_k\leq s\right.\\
& \left.\mbox{ and }(\mu_{i_1},\ldots ,\mu_{i_{k+1}})_m=1\mbox{ for any }1\leq i_1<\ldots <i_{k+1}\leq s\right\},
\end{align*}
where $s,k\in\N$, $s\geq 3$ and $2\leq k\leq s-1$. Note that $\cal{T}_{s,s-1}^m(a_1,a_2,\ldots a_s)=\cal{S}_{s,s-1}^m(a_1,a_2,\ldots a_s)$.

Another possible direction of further study is translating the results from the paper for the case of polynomials of single variable over a field, see \cite{EK}. If $K$ is a field, then the ring $K[t]$ of polynomials of variable $t$ over $K$ is a~Euclidean domain, just as the ring $\Z$ of integers. Dedekind and Weber in \cite{DeWe} proposed considering the problem of representation monic polynomials in $K[t]$ as linear combination of given monic polynomials $A_1,A_2,\ldots,A_s\in K[t]$ with monic coefficients $x_1,x_2,\ldots,x_s\in K[t]$. Concei\c{c}\~{a}o, Gondim and Rodriguez in \cite{CoGoRo} give a short justification of monic polynomials playing the role of positive integers. Namely, every nonzero element of $K[t]$ can be written uniquely as a product of a monic polynomial and a unit in $K[t]$ exactly as every integer can be presented uniquely as a product of a positive integer and a unit in $\Z$. Moreover, they showed that for every monic polynomials $A_1,A_2,\ldots,A_s\in K[t]$ there exists a~constant $g=g(A_1,A_2,\ldots,A_s)$ such that each monic polynomial in $K[t]$ of degree at least $g$ can written as $A_1x_1+A_2x_2+\ldots +A_sx_s$ for some monic $x_1,x_2,\ldots,x_s\in K[t]$. Our results can be extended for the case of the ring $K[t]$ by considering the sets
\begin{align*}
\cal{S}_{s,t}^m(A_1,A_2,\ldots A_s)= & \left\{\sum_{i=1}^sA_ix_i: x_1,\ldots ,x_s\in K[t]\text{ are monic}, (x_{i_1},\ldots ,x_{i_k})_m>1\mbox{ for any }1\leq i_1<\ldots <i_k\leq s\right.\\
& \left.\mbox{ and }(x_1,\ldots ,x_s)_m=1\right\},
\end{align*}
where $m\in\N$, $A_1,A_2,\ldots A_s\in K[t]$ are monic polynomials and $(x_1,\ldots ,x_s)_m$ denotes the common monic divisor of $x_1,\ldots ,x_s$ of greatest degree being an $m$-th power of an element of $K[t]$. The aim would be giving a~constant $c=c(m;A_1,A_2,\ldots A_s)$ such that every monic polynomial in $K[t]$ of degree at least $c$ belongs to the set $\cal{S}_{s,t}^m(A_1,A_2,\ldots A_s)$.

\vskip 1cm

\end{document}